\documentclass[reqno,a4paper,11pt]{amsart}
\usepackage[T1]{fontenc}
\usepackage[utf8]{inputenc}
\usepackage{lmodern}
\usepackage{ amssymb }
 \usepackage{amsthm}
\usepackage{comment}
\usepackage{amssymb,amsmath,amsthm}
\usepackage{xcolor}
\usepackage{graphicx}
\usepackage[all]{xy}
\usepackage{tikz}
\usepackage{relsize}
\usepackage{mhsetup}
\usepackage{mathtools}
\usepackage{float}
\usepackage{stmaryrd}
\usepackage{array}
\usepackage[left=3cm,right=3cm,top=3cm,bottom=3cm]{geometry} % Decrease the margins %
\usepackage[bottom]{footmisc}
\usepackage[colorlinks,citecolor=blue,linkcolor=blue,urlcolor=blue,filecolor=blue,breaklinks]{hyperref}
\usepackage[format=plain,indention=1em,labelsep=quad,labelfont={up},textfont={small},margin=2em]{caption}
\usepackage[mathscr]{euscript}
\usepackage{accents}
\usepackage{footnotebackref}
\newcolumntype{C}[1]{>{\centering\let\newline\\\arraybackslash\hspace{0pt}}m{#1}}
\definecolor{lightblue}{rgb}{0.8,0.8,1}
\numberwithin{equation}{section}
\numberwithin{figure}{section}

\definecolor{vdarkred}{rgb}{0.7,0,0}
\theoremstyle{definition}
    \newtheorem{defn}{Definition}[section]
    \newtheorem{rmk}[defn]{Remark}

    \newtheorem{notation}[defn]{Notation}
 
 \theoremstyle{plain}
    \newtheorem{prop}[defn]{Proposition}
    \newtheorem{lem}[defn]{Lemma}
    \newtheorem{thm}[defn]{Theorem}
    \newtheorem{coro}[defn]{Corollary}

\makeatletter
\renewcommand*{\@seccntformat}[1]{\upshape\csname the#1\endcsname.\hspace{1ex}}
\renewcommand*{\section}{\@startsection{section}{1}{\z@}%
	{2.5ex \@plus 1ex \@minus 0.2ex}%
	{1.5ex \@plus 0.2ex}%
	{\normalfont\normalsize\bfseries}}
\renewcommand*{\subsection}{\@startsection{subsection}{2}{\z@}%
	{2.5ex \@plus 1ex \@minus 0.2ex}%
	{-1.5ex \@plus -0.2ex}%
	{\normalfont\normalsize\bfseries}}
\renewcommand*{\subsubsection}{\@startsection{subsubsection}{3}{\z@}%
	{2.5ex \@plus 1ex \@minus 0.2ex}%
	{-1.5ex \@plus -0.2ex}%
	{\normalfont\normalsize\bfseries}}
\renewcommand*{\paragraph}{\@startsection{paragraph}{4}{\z@}%
	{2.5ex \@plus 1ex \@minus 0.2ex}%
	{-1.5ex \@plus -0.2ex}%
	{\normalfont\normalsize\bfseries}}
\renewcommand*{\subparagraph}{\@startsection{subparagraph}{5}{\z@}%
	{2.5ex \@plus 1ex \@minus 0.2ex}%
	{-1.5ex \@plus -0.2ex}%
	{\normalfont\normalsize\slshape}}
\makeatother

\newcommand{\Q}{\mathbb Q}

\newcommand{\N}{\mathbb N}

\newcommand{\coeff}{\operatorname{Coeff}}
\newcommand{\qptr}{\operatorname{qptr}}
\newcommand{\qt}{\operatorname{qt}}

\newcommand{\Z}{\mathbb Z}
\newcommand{\Li}{{\mathbb L}}
\newcommand{\Liq}{\mathbb L}

\newcommand{\Aut}{\mathrm{Aut}}

\newcommand{\End}{\mathrm{End}}

\definecolor{dgreen}{RGB}{0,150,0}

\newcommand{\incl}[3][right]%
{%
\draw[<-,>=#1 hook] #2 to ($ #2!0.5!#3 $);
\draw[->,>=stealth'] ($ #2!0.5!#3 $) to #3;%
}
\newcommand{\inclusion}[5][right]%
{%
\draw[<-,>=#1 hook] #4 to ($ #4!0.5!#5 $) node[#2,font=\small]{#3};
\draw[->,>=stealth'] ($ #4!0.5!#5 $) to #5;%
}

{\begin{compactitem}

}%
{\end{compactitem}}
{\begin{compactitem}[#1]

}%
{\end{compactitem}}
{\begin{compactdesc}

}%
{\end{compactdesc}}

\newcommand{\cC}{\mathcal{C}}

\newcommand{\cI}{\mathcal{I}}
\newcommand{\cJ}{\mathcal{J}}

\newcommand{\cN}{\mathcal{N}}

\renewcommand{\geq}{\geqslant}
\renewcommand{\leq}{\leqslant}
\renewcommand{\footnoterule}{%
  \kern -3pt
  \hrule width \textwidth height 0.4pt
  \kern 2.6pt
}

\definecolor{dgreen}{RGB}{0,150,0}

\allowdisplaybreaks[4]

\begin{document}
\title[Maximal universal invariants from finite quotients of Verma modules]{Maximal universal invariants of knots from finite quotients of Verma modules}
%\title{Maximal universal invariants from quantum traces on finite quotients of Verma modules}
%\title{Unique quantum traces over quotient rings globalising ADO and coloured Jones polynomials at fixed levels}
\author{Cristina Anghel, Jun Murakami}
\date{}
\begin{abstract}

We construct a sequence of new universal quantum knot invariants that are lifts of both the semi-simple and non semi-simple $U_q(sl_2)$ quantum knot invariants. More specifically, for any level $\cN$ we define a ``level $\cN$ universal invariant'' $ \widetilde{\Omega}_{\cN}(L)$ arising from quantum traces on finite quotients of the generic Verma module over certain quotient rings. We show that for $\cN$ prime, this is the maximal invariant that can arise from the $\cN$-part of the Verma module, and it is a specific  interpolation between the $\cN^{th}$ coloured Jones and $\cN^{th}$ ADO polynomials. For $\cN$ non prime $ \widetilde{\Omega}_{\cN}(L)(q,s)$ has a richer structure, it recovers the  $\cN^{th}$ coloured Jones and $\cN^{th}$ ADO polynomials, but it could contain more information which is not seen in the sequence of coloured Jones and ADO invariants. 
\end{abstract}
%{\tableofcontents}

\maketitle
{
\makeatletter
\renewcommand*{\BHFN@OldMakefntext}{}
\makeatother
}
\vspace{-5mm}

\section{Introduction}\label{introduction}

We study quantum invariants arising from the infinite dimensional representation theory of the quantum group $U_q(sl_2)$. 
The family of finite dimensional representations of this quantum group at generic $q$ leads to the sequence of coloured Jones polynomials whereas the  representations at roots of unity give the sequence of coloured Alexander polynomials (or ADO invariants \cite{ADO}). 
The asymptotic behaviour of these invariants is subject to important questions in low dimensional topology: the Volume Conjecture (Kashaev \cite{K},\cite{M2}) and Gukov-Manolescu's Conjecture (\cite{GM}) which predict geometric information that is contained in their limit with respect to the colour. 
Motivated by the question of the conjectured geometry contained in these invariants the second author provided a unified topological framework that encodes all coloured Jones and ADO polynomials at bounded colours by graded intersections in configuration spaces (\cite{Cr1}, \cite{Cr2}). They lead to universal geometrical invariants for links (\cite{Cru1}, \cite{Cru2}).  These unification results opened up questions about encoding non-semisimple invariants and semi-simple invariants directly from a unified perspective. In \cite{CrG} the author defined a globalisation of Jones and Alexander invariants by constructing a graded intersection that is taking values in a  quotient of a polynomial ring in two variables. 

{\bf Main results} In this paper we address this problem at higher colours, and we want to unify coloured Jones and coloured Alexander invariants corresponding to $\cN$ dimensional representations of $U_q(sl_2)$. 
Our strategy is to use quantum groups over two variables together with their braid group action directly on the infinite dimensional representation, which is given by the Verma module. We will work over the Laurent polynomial ring $\Liq=\Q[q^{\pm2}, s^{\pm 2}]$.

{\bf Question:} For a fixed level $\cN$, which is the maximal quotient of $\Liq$ in which we can obtain invariants from the $\cN$-dimensional part of the $U_q(sl_2)$-Verma module?
We start from the generic braid action on the Verma module and define quantum traces on specialisations of the $\cN$-dimensional part over certain quotient rings. We prove the following:
\par
\begin{itemize}
\item[1)] {\color{blue} \em Level $\cN$ unified invariant at prime levels.} If $\cN$ is prime, then the maximal quotient ring of $\Liq$ in which we get invariants from Verma modules is given by the ideal $\cI_\cN$ (as in \eqref{qideal}, Corollary \ref{coro:prime}).
%\par
The quantum trace in this quotient leads to a 
``Level $\cN$ unified invariant'' (for any level $\cN$):
$${\color{black} \Omega_{\cN}(L)(q,s) \in \Li_{\cN}}.$$
\item[2)] {\color{blue} \em Level $\cN$ universal invariant.} If $\cN$ is not prime, then the maximal quotient ring in which we get invariants from Verma modules is given by the ideal $\widetilde{\cI}_{\cN}$ (see \eqref{idealm}, Theorem \ref{thm:THEOREM}). 
The intersection form in this quotient leads to a maximal invariant called ``Level $\cN$ universal invariant'':  $$ \widetilde{\Omega}_{\cN}(L)(q,s) \in \widetilde{\Li}_\cN.$$

\item[3)] {\color{blue} \em Interpolation formula in the level $\cN$ quotient ring.} The level $\cN$ unified invariant $\Omega_{\cN}(L)(q,s)$ is an interpolation between the coloured Jones and ADO invariants at level $\cN$, which we compute explicitely in Theorem \ref{TINT}.
\end{itemize}

\subsection{Quotient rings} 

Let us fix a natural number $\cN \in \mathbb N$, $\cN \geq 2$. 
\begin{defn}[Level $\cN$-quotient ideals] We are going to define two ideals in the ring $\Liq$, and work over the associated quotient rings, as follows.

1) (Interpolation quotient ideal) Let
\begin{equation}\label{qideal}
\cI_\cN:=\left(\varphi_\cN(q^2)\cdot (s^2q^{-2(\cN-1)}-1)\right)
\subseteq \Liq=\Q[q^{\pm2}, s^{\pm 2}]
\end{equation}
and refer to it as the ``interpolation quotient ideal''.

2) (Maximal quotient ideal) Also, we define 
\begin{equation}\label{idealm}
\widetilde{\cI}_\cN:=\left( \varphi_\cN(q^2) \cdot (s^2 q^{-2\cN+2} - 1)\right)  \cap \prod_{\text{\scriptsize$\begin{matrix}
 d\mid \cN \\ d \neq 1, \cN \end{matrix}$}}
 \left(\varphi_d(q^2),\ \  \frac{s^{2d}-1}{s^2-1} \right)
\subseteq \Liq=\Q[q^{\pm2}, s^{\pm 2}]
\end{equation}
and call it the ``maximal quotient ideal''. Here, $\varphi_\cN(q)$ is the $\cN^{th}$ cyclotomic polynomial.
\end{defn}
\begin{defn}[Quotient rings] We define the following two  quotient rings of $\Liq$.

1) The ``interpolation ring'', given by:
\begin{equation}
\Li_{\cN}:=\Q[q^{\pm2}, s^{\pm 2}] \ / \ \cI_\cN.
\end{equation}
2) The ``maximal ring'', given by
\begin{equation}\label{qring}
\widetilde{\Li}_{\cN}:=\Q[q^{\pm2}, s^{\pm 2}] \ / \widetilde{\cI}_\cN.
\end{equation}
\end{defn}
Note that there is the natural projection $\pi_{\cN} : \widetilde{\Li}_{\cN} \to {\Li}_{\cN}$ since $\widetilde{\cI}_\cN \subset {\cI}_\cN$, and $\widetilde{\Li}_{\cN} = {\Li}_{\cN}$ if $\cN$ is a prime number.  

\subsection{Level $\cN$ representation}

We are going to use the $U_q(sl_2)$ Verma module $\hat{V}$ (defined over $\Li$), which is generated by an infinite family of vectors $\{v_0, v_1,...\}$. 
Using the two variable $R$-matrix associated to this quantum group given by \eqref{eq:Rpositive} and \eqref{eq:Rnegative},
 we get a braid group action:
\begin{equation}
\hat{\rho}_n: B_n \rightarrow \Aut_{U_q(sl_2)}\left(\hat{V}^{\otimes n}\right). 
\end{equation}
We fix a level $\cN$ and we want to use the $\cN$-dimensional $\Li$-submodule generated by the first $\cN$ vectors of the Verma module. 
We refer to this module as the ``$\cN$-finite part from the Verma module''. 
We introduce the following modules. 
\begin{defn}[Quotient modules]
Let us denote the above specialisations of the $\cN$-finite part from the Verma module over the interpolation ring and over the maximal ring respectively: 
\begin{equation}
\begin{aligned}
& 1) \text{ Interpolation module}: \ \ \ \ \
V_{\cN}:=\langle v_{0},...,v_{\cN-1}\rangle_{\Li_{\cN}} \subseteq \hat{V}\otimes_{\Li}{\Li_{\cN}}\\
& 2)  \text{ Maximal module}: \ \ \ \ \ \ \ \ \ \ \ \widetilde{V}_{\cN}:=\langle v_{0},...,v_{\cN-1}\rangle_{\widetilde{\Li}_{\cN}} \subseteq \hat{V}\otimes_{\Li}{\widetilde{\Li}_{\cN}}.
\end{aligned}
\end{equation}
\end{defn}
We obtain the following. 
\begin{thm}[Action on quotients of specialised Verma module over the quotient rings] 
The generic braid group action $\hat{\rho}_n$ leads to braid group representations on the quotient modules as below:
$$1) \ \rho_{\cN,n}  :B_n \rightarrow \Aut_{\Li_{\cN}} \left( V_{\cN}^{\otimes n}\right) $$
$$2) \ \widetilde{\rho}_{\cN,n}  :B_n \rightarrow \Aut_{\widetilde{\Li}_{\cN}} \left( \widetilde{V}_{\cN}^{\otimes n}\right). $$
\end{thm}

\subsection{Unique quantum traces over quotient rings}
We want to define knot invariants from quantum traces that are constructed on the endomorphism rings of these modules. 
First, we will prove that there exists a unique quantum trace associated the level $\cN$ maximal module.
 
\begin{defn}
A map $\qptr_{\widetilde{V}_{\cN}^{\otimes n}} : \End(\widetilde{V}_{\cN}^{\otimes (n+1)})\to \End(\widetilde{V}_{\cN}^{\otimes n})$
defined by
\[
\qptr_{\widetilde{V}_{\cN}^{\otimes n}}(f)_{v_{i_1,i_2, \cdots, i_n}}^{v_{j_1,j_2, \cdots, j_n}}
:=\sum_{j=0}^{\cN-1}
g_j(q, s) \cdot f_{v_{i_1,i_2, \cdots, i_n,j}}^{v_{j_1,j_2, \cdots, j_n,j}}
\]
for some scalars $g_i(s, q)$
 is called a ``quantum partial trace'' if there are two non-zero scalars $a$, $b$ satisfying
\begin{equation}
\qptr_{\widetilde{V}_{\cN}^{\otimes n}}(\widetilde\rho_{\cN,n}(\sigma_n\beta_n)) 
=
a \widetilde\rho_{\cN,n}(\beta_n), 
\quad
\qptr_{\widetilde{V}_{\cN}^{\otimes n}}(\widetilde\rho_{\cN,n}(\sigma_n^{-1}\beta_n) )
=
a^{-1} \widetilde\rho_{\cN,n}(\beta_n)
\end{equation}
 for any braid $\beta_n \in B_n$ {(see also Notation \ref{endcoeff})}.  
\end{defn}

  \begin{thm}[Construction of a unique quantum trace over quotient rings]

There exists a unique quantum trace, up to a scalar, defined on the tensor powers of the maximal module (Theorem \ref{thm:uqtr1}):
$$
\qptr_{\widetilde{V}_{\cN}^{\otimes n}}:\operatorname{End}(\widetilde{V}_{\cN}^{\otimes (n+1)})\rightarrow \operatorname{End}(\widetilde{V}_{\cN}^{\otimes n}),
$$ 
where $g_j(q, s) = g_0(q, s) q^{-2j}$ for some non-zero $g_0(q, s)$.  
\end{thm}

\subsection{Level $\cN$ invariants in quotient rings} 
We are going to use the above quantum trace, in order to introduce the level $\cN$ unified invariant of knots.  
Let 
${g_0}(q, s) = s^{-(\cN-1)}q^{(\cN-1)}$ and
\[
\qt_{\widetilde{V}_{\cN}}:=\qptr_{\widetilde{V}_{\cN}} \circ \qptr_{\widetilde{V}_{\cN}^{\otimes 2}} \circ \cdots \circ \qptr_{\widetilde{V}_{\cN}^{\otimes(n-1)}} : \operatorname{End}(\widetilde{V}_{\cN}^{\otimes n})\rightarrow \operatorname{End}(\widetilde{V}_{\cN}).
\]
Let $K$ be a knot and $\beta_n$ is a $n$ braid whose closure is isotopic to $K$.  
According to the  standard argument to construct a knot invariant from the quantum $R$ matrix, $\qt_{\widetilde{V}_{\cN}}(\rho_{\cN, n}(\beta_n))$   is a scalar matrix, and this scalar normalized by the writhe gives a knot invariant.  
\begin{defn}[Level $\cN$ unified invariant]
Let
\begin{equation}
\widetilde\Omega_{\cN}(\beta_n)(q,s)
:= 
s^{w(\beta_n)(\cN-1)}q^{-w(\beta_n){\cN(\cN-1)}} \cdot 
\qptr_{\widetilde{V}_{\cN}}\left( \rho_{\cN,n}
\right)_{v_0}^{v_0} \in \widetilde{\Li}_{\cN}.
\end{equation}
\end{defn}
Our main result is that $\widetilde\Omega_{\cN}(\beta_n)(q,s)$ is a universal knot invariant coming from the level $\cN$ representations of $U_q(sl_2)$.  
\begin{thm}[{\bf Level $\cN$ maximal universal invariant}]\label{thm:THEOREM}
For any $\cN \in \N$, $\cN \geq 2$, $\widetilde{\Li}_{\cN}$ is the largest quotient of the polynomial ring $\Liq$ with the property that the image of the quantum trace in this quotient gives a knot invariant. 
\par
More precisely, the image of the level $\cN$ quantum trace in this quotient ring $\widetilde{\Omega}_{\cN}(L)(q,s) \in \widetilde{\Li}_{\cN}$ is a knot invariant. 
Also, if $\widetilde{\Li}_{\cN}'$ is a quotient of $ \Li$ in which $\widetilde{\Omega}_{\cN}(\beta_n)$ is a knot invariant, then the quotient onto $\widetilde{\Li}_{\cN}'$ factors through $\widetilde{\Li}_{\cN}$.
\end{thm}

If $\cN$ is a prime integer, then $\widetilde{\Li}_{\cN} = \Li_{\cN}$ and this theorem implies the following.  

\begin{coro}[{\bf Level $\cN$ maximal unified invariant at prime parameters}]
\label{coro:prime}
If $\cN \in \mathbb N$ is prime, the ring $\Liq_{\cN}$ is the largest quotient of the polynomial ring $\Li$ with the property that the image of the quantum trace in this quotient gives a knot invariant. 

More precisely, for any $\cN \in \cN$, $\cN\geq 2$ we have that 
\[
\Omega_{\cN}(L)(q,s) = \pi_{\cN}(\widetilde{\Omega}_{\cN}(\beta_n)) \in \Li_{\cN}
\]
 is a well-defined oriented knot invariant. Also, if $\cN$ prime and we consider $\Li_{\cN}'$ a quotient of $ \Li$ in which $\Omega_{\cN}(\beta_n)$ is a knot invariant, then the quotient onto $\Li_{\cN}'$ factors through $\Li_{\cN}$.
\end{coro}

\begin{figure}[h]
\begin{center}
\hspace*{-6mm}\begin{tikzpicture}
[x=1.2mm,y=1.4mm]
% Nodes of the diagram

\node (Jn)               at (-45,-2.4)    {$\begin{matrix}\\[-10pt]J_{\cN}(L)\end{matrix}$};

\node (Am)               at (-12,-2.4)    {$\begin{matrix}
\\[-10pt]{\Phi_{\cN}}(L) \end{matrix}$};

\node(IA) at (0,-10) {Coloured Alexander polynomials};
\node(IJ) at (-54,-10) {Coloured Jones polynomials};
\node(IA11) at (-50,20) {Different at non-prime levels};
\node(IJ11) at (-24,20) {$\pi_{\cN}$};
\node(IJ11) at (2,20) {The same for $\cN$ prime Cor.\ref{coro:prime}};

\node(LI)[draw,rectangle,inner sep=3pt,color=blue] at (-27,12) {\begin{tabular}{l}
 \ \ \ \ \ Level $\cN$ \\ unified invariant\\ \ \    ${\color{black} \Omega_{\cN}(L)(q,s) \in \Li_{\cN}}$
\end{tabular}};

\node(LIM)[draw,rectangle,inner sep=3pt,color=blue] at (-27,30) {\begin{tabular}{l}
 \ \ \ \ \ Level $\cN$ \\ universal invariant \\ \ \    ${\color{black} \widetilde{\Omega}_{\cN}(L)(q,s) \in \widetilde{\Li}_{\cN}}$
\end{tabular}};

\node(L)[draw,rectangle,inner sep=3pt] at (-27,47) {\begin{tabular}{l}
 \ \ \ \ \ \ Braid action \\ on the Verma module \end{tabular}};
\node(IA) at (10,45) {$\Liq=\Q[q^{\pm2}, s^{\pm 2}] $}; 
\node(IA) at (10,28) {$\widetilde{\Li}_{\cN}=\Q[q^{\pm2}, s^{\pm 2}] \ / \widetilde{\cI}_\cN$}; 
\node(IA) at (10,10) {$\Li_{\cN}=\Q[q^{\pm2}, s^{\pm 2}] \ / \ \cI_\cN$};
\node(IA) at (-70,33) {$\widetilde{\cI}_\cN=\left( \varphi_\cN(q^2)\cdot(s^2 q^{-2\cN+2} - 1)\right) \cap$}; 
\node(IA) at (-70,27) {$\cap \prod_{\text{\scriptsize$\begin{matrix}
 d\mid \cN \\ d \neq 1, \cN \end{matrix}$}}
 \left(\varphi_d(q^2), \frac{s^{2d}-1}{s^2-1} \right)
$}; 

\node(IA) at (-69,11) {$\cI_\cN=\left(\varphi_\cN(q^2)  (s^2q^{-2(\cN-1)}-1)\right)
$};

\node(Q)[draw,rectangle,inner sep=3pt,color=red, minimum width = 2cm, 
    minimum height = 1cm] at (-45,-3) {$\phantom{A}$};
\node(Q)[draw,rectangle,inner sep=3pt,color=red, minimum width = 2cm, 
    minimum height = 1cm] at (-13,-3) {$\phantom{A}$};
   \draw[->]             (LI)      to node[right,xshift=2mm,font=\large]{}   (Am);
      \draw[->]             (LI)      to node[right,xshift=2mm,font=\large]{}   (Jn);
 \draw[->]             (LIM)      to node[right,xshift=2mm,font=\large]{}   (LI);
   \draw[->]             (L)      to node[right,xshift=10mm,yshift=-9mm,font=\large]{ \hspace{10mm}{\normalsize Quotient ring  \ Th.\ref{thm:THEOREM}}}   (LIM);
   \draw[->]             (LIM)      to node[right,xshift=10mm,yshift=-10mm,font=\large]{ \hspace{8mm} {\normalsize Interpolation formula Th.\ref{TINT}}}   (LI);
\end{tikzpicture}
\end{center}
\caption{Invariants in quotient rings}
\end{figure}
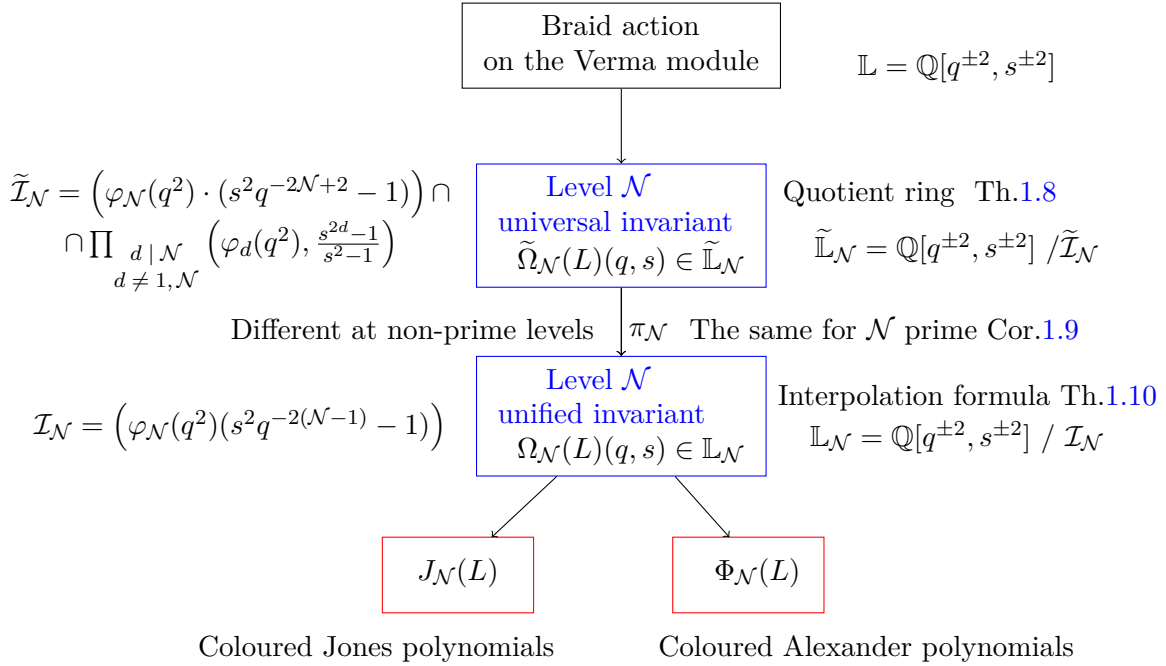
Our next result aims to answer the question about the nature of the two invariants that we constructed so far. 
 \subsection{Universal versus unified invariants}
 In subsection \ref{rec} we study relations between these invariants, and we show in Theorem \ref{JA3} that the level $\cN$ universal invariant recovers the level $\cN$ unified invariant. Then in Theorem  \ref{JA1} and Theorem \ref{JA3} we show that both level $\cN$ universal and unified invariant specialise to the $\cN^{th}$ coloured Jones and ADO polynomials by change of coefficients. 
 \par
 Then, we show that for any $\cN$ (not necessary a prime number) the level $\cN$ invariant $\Omega_{\cN}(L)(q,s)$ is an interpolation between the coloured Jones and ADO invariants at level $\cN$, as below.
 \begin{thm}[{\bf Level $\cN$ unified invariant interpolates coloured Jones and ADO invariants}]\label{TINT} We have that  $\Omega_{\cN}(L)(q,s) \in \Li_{\cN}$ is a well defined oriented knot invariant and globalises the $\cN^{th}$ coloured Jones and coloured Alexander invariants, as follows: 
 \begin{equation}
\Omega_{\cN}(L)(q,s) = J_{\mathcal{N}}(L, q) + \Phi_{\mathcal{N}}(L, s) - \Phi_{\mathcal{N}}(L, q^{1-\mathcal N}).
\end{equation}
 \end{thm}
As we have seen above, the level $\cN$ invariant $\Omega_{\cN}(L)(q,s) \in \Li_{\cN}$ has a clear structure, providing a globalisation that interpolates coloured Jones and ADO invariants at level $\cN$. 
\begin{rmk}[Rich structure of the level $\cN$ universal invariant] 
The level $\cN$ universal invariant  $\widetilde{\Omega}_{\cN}(L)(q,s) \in \widetilde{\Li}_{\cN}$ has a richer structure, belonging to maximal ring which surjects onto the interpolation ring. We have well-defined specialisation maps from $\widetilde{\Li}_{\cN}$ to the rings where the $\cN^{th}$ coloured Jones and ADO invariants belong, so $\widetilde{\Omega}_{\cN}(L)(q,s) $ recovers $J_{\cN}(L)$ and $\Phi_{\cN}(L)$ by specialisation of coefficients. 
\par
On the other hand, the maximal ring $\widetilde{\Li}_{\cN}$ is given by the ideal 
$$
\widetilde{\cI}_\cN=\left( \varphi_\cN(q^2) \cdot (s^2 q^{-2\cN+2} - 1)\right)  \cap \prod_{\text{\scriptsize$\begin{matrix} d\mid \cN \\ d \neq 1, \cN \end{matrix}$}}\left(\varphi_d(q^2), \frac{s^{2d}-1}{s^2-1} \right) \subseteq \Liq.
$$
If we fix $d \mid \cN$, the presence of the ideal
$$
\left(\varphi_d(q^2), \frac{s^{2d}-1}{s^2-1}\right)
$$ 
suggests that $\widetilde{\Omega}_{\cN}(L)(q,s) $ is related to ADO invariants at level $d$ and an invariant arising from Verma modules at the highest weight parameter being the $d^{th}$ root of unity. This shows that this $d$-part of the ideal, generated by both terms $\varphi_\cN(q^2)$ and $\frac{s^{2d}-1}{s^2-1}$  has a rich structure which in turn suggests that the invariant itself could encode deeper information than the sequence of coloured Jones and ADO invariants at levels that divide $\cN$.
\end{rmk}
\subsection{Further directions}
\subsubsection{Structure of the level $\cN$ universal invariant}  We would like to study the form of the maximal ring at non-prime levels $\widetilde{\Li}_{\cN}$ and the information encoded in the maximal level $\cN$ universal invariant at  $\widetilde{\Omega}_{\cN}(L)(q,s) \in \widetilde{\Li}_{\cN}$.
\subsubsection{Geometrical Level $\cN$ universal invariant} 
In \cite{Cr2}, \cite{Cr1} the author constructed universal knot invariants and unifications of coloured Jones and ADO polynomials using topological tools, based on configuration spaces in the punctured disc. The level $\cN$ universal and level $\cN$ unified invariant should have purely topological counterparts in that context, given by Lagrangian intersections in configuration spaces. We are planning to investigate such models for our future research.
\subsubsection{Level $\cN$ universal invariants for $3$-manifolds} 
This direction concerns the lift of the level $\cN$ universal knot invariants to invariants for $3$-manifolds. Coloured Jones and ADO knot invariants are building blocks for the Witten-Reshetikhin-Turaev invariants (\cite{Witt}, \cite{RT}) and the Costantino-Geer-Patureau invariant (\cite{BCGP}) respectively. We plan to investigate the existence of $3$-manifolds invariants at level $\cN$ that unify the Witten-Reshetikhin-Turaev invariant and the Costantino-Geer-Patureau invariant.
\subsection*{Acknowledgements} 
 The first author gratefully acknowledges the support of the ANR grant ANR-24-CPJ1-0026-01 at Universit\'e Clermont Auvergne - LMBP. Also, she acknowledges partial support by grants of the Ministry of Research, Innovation and Digitization, CNCS - UEFISCDI, project numbers  PN-IV-P2-2.1-TE-2023-2040 and PN-IV-P1-PCE-2023-2001, within PNCDI IV. 
The second author gratefully acknowledges  support by JSPS KAKENHI Grant Number JP23K20214. 
\section{Quantum group ${U}_q(sl_2)$}
\subsection{Notations}\label{Not}
We will use the following notations for quantum numbers:
$$ 
\{ x \} =q^x-q^{-x}, \ \ \ \
 [x]_{q}= \frac{q^x-q^{-x}}{q-q^{-1}}, \ \ \ \ \ \ \ 
 [n]_q!=[1]_q \cdot [2]_q \cdot \ldots \cdot [n]_q,
$$
$$
{n \brack j}_q=\frac{[n]_q!}{[n-j]_q! [j]_q!}.
$$
\begin{notation}(Specialisation of coefficients)\label{N:spec}\\ 
Let us consider two rings $R,S$ and a free $R$-module $M$. We suppose that we have a fixed basis $\mathscr B$ of $M$. We suppose that we have a specialisation of coefficients that is given by a ring morphism:
$\psi: R \rightarrow S.$
Then, we define the specialisation of $M$ by the change of the coefficients $\psi$ as the following $S$-module: $$M|_{\psi}:=M \otimes_{R} S$$ which has the corresponding basis given by 
$$
\mathscr B_{M|_{\psi}}:=\mathscr B \otimes_{R}1 \in M|_{\psi}. 
$$
\end{notation}
\begin{notation}[Rings and quotient rings] 
We will use the Laurent polynomial rings, denoted as below:
\[
\Li^{\Z}=\mathbb Z[q^{\pm2}, s^{\pm 2}],  \qquad
\Liq=\mathbb Q[q^{\pm2}, s^{\pm 2}].
\]
\end{notation}
{
\begin{notation}(Endomorphism coefficients) \label{endcoeff}Let us consider a set of indices 
$$
i_1,...,i_{n} \in \{ 0, \cdots, \cN-1\}.
$$
We define the associated vector: 
$$
v_{i_1,...,i_n}:=v_{i_1}\otimes \cdots v_{i_n} \in \widetilde{V}_{\cN}^{\otimes n}.
$$
Also, for an endomorphism $f \in \operatorname{End}(V_{\cN}^{\otimes n})$ we denote by:
\begin{equation}
f_{v_{i_1,i_2, \cdots, i_n}}^{v_{j_1,j_2, \cdots, j_n}}:=\coeff(v_{j_1,j_2, \cdots, j_n}, f(v_{i_1,i_2, \cdots, i_n})).
\end{equation}
Here, for a fixed basis of a module $V_{\cN}^{\otimes n}$ and two basis vectors $v$, $w$ we denote by 
$$\coeff(v,w)$$ 
the coefficient of the vector $v$ in the decomposition of $w$ with respect to the fixed basis.
\end{notation}}
\subsection{Verma module of $U_q(sl_2)$}\label{2}
%
%\subsection{$U_q(sl_2)$ and its representations}
Let us fix $q, s$ to be parameters and in the next part we are going to work over the ring 
$\Li$.
We will use the version of the quantum group $U_q(sl_2)$ with divided powers, as presented in \cite{JK}.
\begin{defn}[Quantum group over two variables]
Let  $U_q(sl_2)$ be the algebra over $\Li$ that is generated by the sequence of elements $\{ E,\, F^{(n)}, \,K^{\pm1} \mid n \in \mathbb N\}$ subject to the following relations:
\begin{equation*}
\begin{aligned}
&KK^{-1}=K^{-1}K=1; \ \ \ KE=q^2EK; \ \ \ KF^{(n)}=q^{-2n} F^{(n)}K;\\
&F^{(n)}F^{(m)}= {n+m \brack n}_{q} F^{(n+m)};\\
&[E,F^{(n+1)}]=F^{(n)}(q^{-n}K-q^{n}K^{-1}).
\end{aligned}
\end{equation*}
\end{defn}
Then  $U_q(sl_2)$ has a structure of a Hopf algebra (as in \cite{JK}).
\begin{defn}[Verma module]
We define $\hat{V}$ to be the $\Li$-module that is generated by a sequence of vectors that we denote as $\{v_0, v_1,...\}$. Then, we consider the following $U_q(sl_2)$ action on this module $\hat{V}$ (see \cite{JK}):
\begin{equation*}%\label{action}  
\begin{cases}
Kv_i=sq^{-2i}v_i,\\
Ev_i=v_{i-1},\\
\displaystyle F^{(n)}v_i = {n+i \brack i}_q \prod _{k=0}^{n-1} (sq^{-k-i}-s^{-1} q^{k+i}) v_{i+n}.
\end{cases}
\end{equation*}
\end{defn}
Further, using the structure of the Hopf algebra, we have a braiding that leads to the following $R$-matrix.
\begin{prop}[Generic R-matrix action (\cite{JK})] \label{R:R}
We have the following action on the basis of the tensor power $\hat{V}\otimes \hat{V}$:
\begin{equation}\label{eq:Rpositive}
R(v_i\otimes v_j)= s^{-(i+j)} \sum_{n=0}^{i} q^{2(i-n)(j+n)}q^{\frac{n(n-1)}{2}} {n+j \brack j}_q \prod_{k=0}^{n-1}(sq^{-k-j}-s^{-1}q^{k+j}) \ v_{j+n}\otimes v_{i-n}
 \end{equation}
 \begin{equation}\label{eq:Rnegative}
 R^{-1}(v_i\otimes v_j)= s^{(i+j)} \sum_{n=0}^{j} (-1)^{n}q^{-2ij}q^{-\frac{n(n-1)}{2}} {n+i \brack i}_q \prod_{k=0}^{n-1}(sq^{-i-k}-s^{-1}q^{k+i}) \ v_{j-n}\otimes v_{i+n}.
 \end{equation}
Then, this will lead to a braid group representation, which we denote as follows:
\begin{equation}
\begin{aligned}
\hat{\rho}_n: B_n \rightarrow \Aut_{U_q(sl_2)}\left(\hat{V}^{\otimes n}\right), \ \ \ \ \ \ \ \ \ \ \ \ \ \\
\ \ \ \ \ \ \ \ \ \ \ \ \ \ \ \ \ \ \sigma _i^{\pm 1}\mapsto Id_V^{\otimes (i-1)}\otimes  R^{\pm 1}  \otimes Id_V^{\otimes (n-i-1)}.
\end{aligned}
\label{eq:action}
\end{equation}
\end{prop}
\section{Maximal quotient rings}
In this section, we introduce various expressions of the  ideal $\widetilde\cI_{\cN}$ in $\Li$ to define the maximal quotient ring $\widetilde{\Li}_{\cN}$.  
\subsection{Defining ideals}
\begin{defn}[Condition ideal] 
\label{def:condition} 
Let us consider the following ideals in $\Liq$: 
%
%\begin{equation}
%\cI_{\cN}:=\Big( \varphi_\cN(q^2)\cdot (s^2q^{-2\cN+2}-1) \Big),
%\end{equation}
%
\begin{equation*}
\cJ_{\cN,k}:=\left( \varphi_k(q^2),\  \ \prod_{i=0}^{k-2}(s^2q^{-2i}-1) \right)
\end{equation*}
for $k| \cN$, $k \neq 1$, $2$, $\cN-1$, $\cN$ where $\varphi_k$ is the $k^{th}$ cyclotomic polynomial, and
\begin{equation*}
\widetilde{\cJ}_\cN = \cI_\cN \cap\!\!\!\! \bigcap_{\text{\scriptsize$\begin{matrix}
d | \cN\\[-1pt]
d\neq 1, \cN
\end{matrix}$}} \cJ_{\cN,d},  
\end{equation*}
where $\cI_\cN = \Big( \varphi_\cN(q^2)\cdot (s^2q^{-2\cN+2}-1) \Big)$ given by \eqref{qideal}.  
We call $\widetilde{\cJ}_\cN$ the ``condition ideal''.
\end{defn}
\begin{defn}[Maximal quotient ideal]
\label{def:maximal}
Let us consider as well the ideal
\begin{equation}\label{idealb}
\cC_\cN = \sum_{\substack{i,j=0 \\ j+n\geq \cN,0\leq  n\leq i}}^{\cN-1}
 \left( {n+j \brack j}_q \prod_{k=0}^{n-1}(sq^{-k-j}-s^{-1}q^{k+j}) \right) \subseteq \Liq_{\cN}
\end{equation}
and call it the ``maximal quotient ideal''.
\end{defn}
\subsection{Coincidence of defining ideals}
We  show that the three ideals $\widetilde{\cJ}_{\cN}$, $\cC_{\cN}$ and $\widetilde{\cI}_\cN$ are equal.  
\begin{prop}[Structural inclusion of condition ideals]\label{prop:EQ}
The  above three mentioned structural ideals coincide and they also coincide with $\widetilde\cI_\cN$, i.e.
\begin{equation}\label{eq:EQ}
\cC_\cN = \widetilde{\cJ}_{\cN} = \widetilde\cI_\cN. 
\end{equation}
\end{prop}
The proof of this proposition is given in Appendix \ref{Ap}.  
\begin{coro}
If $\cN$ is prime, then 
\begin{equation}\label{eq:EQp}
\cC_\cN = \widetilde{\cJ}_{\cN}  = \widetilde\cI_\cN= \cI_\cN.
\end{equation}
\end{coro}
\begin{rmk}
The equality \eqref{eq:EQ} holds when working over rational coefficients, this is not true if the ground ring is $\mathbb{Z}[q^{\pm1}, s^{\pm1}]$ instead of $\Liq=\mathbb{Q}[q^{\pm1}, s^{\pm1}]$.
On the other hand, the equality \eqref{eq:EQp} also holds over the ground ring $\mathbb{Z}[q^{\pm1}, s^{\pm1}]$.
\end{rmk}
\subsection{Maximal quotient ring at level $\cN$}
\begin{defn}[Quotient ring at non prime levels] 
\label{def:ring}
Let us consider the following quotient ring, and call it ``maximal quotient ring'':
\begin{equation*}
\widetilde{\Li}_{\cN}=\Q[q^{\pm2}, s^{\pm 2}] \ / 
{\widetilde{\cI}_\cN}=
\Q[q^{\pm2}, s^{\pm 2}] \ / 
{\cC_\cN}.
\end{equation*}
\end{defn}
\begin{rmk}
If $\cN$ is prime, then
\[
\widetilde\Li_\cN = \Li_\cN = \Li/\cI_\cN.
\]
\end{rmk}
\section{Level $\cN$ universal invariant}
Our aim is to obtain knot invariants at level $\cN$, in this case for any natural number $\cN$ (not necessary prime as in the previous sections). In the following sections we will see that if $\cN$ is not prime, we are able to obtain invariants in rings that are higher than the ring that we have defined for the prime case, which is $\Li_{\cN}$. 
\par
We do so by looking at the $\cN$-dimensional submodule arising from the generic Verma module, when specialised over this higher order quotient ring.
\begin{notation}($\cN$-dimensional submodule over the quotient ring)\label{mod2}
Let us consider the specialisation the $\cN$-dimensional submodule in the Verma module $V_{\cN}$ specialised over the level $\cN$ quotient ring as below: 
 $$\widetilde{V}_{\cN}:=\langle v_{0},...,v_{\cN-1}\rangle_{\widetilde{\Li}_{\cN}} \subseteq \hat{V}\mid_{\widetilde{\Li}_{\cN}}.$$
\end{notation}
Our aim is to show that we can obtain knot invariants in $\widetilde{\Li}_{\cN}$. In order to do so we are going to use braid group actions on tensor powers of $\widetilde{V}_{\cN}$. 
\subsection{Braid action over the quotient ring}\label{action2}
 First we will see that the specialised $R$-matrix at these higher rings leads to a well defined action onto $\widetilde{V}_{\cN}^{\otimes n}$.
\begin{prop}[Action on the higher level $\cN$ module over the quotient ring] 
The action $\hat{\rho}_n$ given by \eqref{eq:action} specialized to the level $\cN$ quotient ring $\widetilde\Li_\cN$ leads to a well defined induced braid group action:
$$ \widetilde{\rho}_{\cN,n}  :B_n \rightarrow \Aut_{\widetilde{\Li}_{\cN}} \left( \widetilde{V}_{\cN}^{\otimes n}\right). $$
\end{prop}
\begin{proof}
 The condition that we want to prove is that the action $\widetilde{\rho}_{\cN,n}$ on $\hat{V}\mid_{\widetilde{\Li}_\cN}^{\otimes n}$ preserves the vector subspace $\widetilde{V}_{\cN}^{\otimes n}$. 
Following the definition of the action $\hat{\rho}_{n} $, it is enough prove that
the action of $R$ preserves the vector subspace ${\widetilde{V}_{\cN}^{\otimes 2}}$.
We will show the following property:
 $$
 R\mid_{\widetilde{\Li}_{\cN}}(v_i\otimes v_j) \in \widetilde{V}_{\cN}^{\otimes 2},  \ \ \ \forall i,j \in \{0,...,\cN-1\}.
 $$
Let $i,j \in \{0,...,\cN-1\}$. 
The $R$-matrix action is given by \eqref{eq:Rpositive} as follows:
 \begin{equation*} 
R(v_i\otimes v_j)= s^{-(i+j)} \sum_{n=0}^{i} q^{2(i-n)(j+n)}q^{\frac{n(n-1)}{2}} {n+j \brack j}_q \prod_{k=0}^{n-1}(sq^{-k-j}-s^{-1}q^{k+j}) \ v_{j+n}\otimes v_{i-n}.
 \end{equation*}
 We recall that we work over the quotient ring 
 $\widetilde{\Li}_{\cN}=
\Q[q^{\pm2}, s^{\pm 2}] \ / {\widetilde{\cI}_\cN}$, 
and this is identified with 
$\widetilde{\Li}_{\cN}=
\Q[q^{\pm2}, s^{\pm 2}] \ / {\cC_\cN}$ 
using Definition \ref{def:ring}. 
Let us also recall the definition of the ideal $\cC_\cN$ from relation \eqref{idealb}: 
\begin{equation*}
\cC_\cN= \sum_{\substack{i,j=0 \\ j+n\geq \cN,0\leq  n\leq i}}^{\cN-1}
 \left( {n+j \brack j}_q \prod_{k=0}^{n-1}(sq^{-k-j}-s^{-1}q^{k+j}) \right) \subseteq \Liq_{\cN}
\end{equation*}
Let 
$$
\widetilde{\pi}: \Liq \rightarrow \widetilde{\Li}_{\cN}=
\Q[q^{\pm2}, s^{\pm 2}] \ / {\cC_\cN}.
$$ 
be the canonical projection. 
We are going to prove that for any indices $i,j \leq \cN-1 \text { and } j+n\geq \cN$,  we have that:
\begin{equation*}
\widetilde{\pi} \left( {n+j \brack j}_q \prod_{k=0}^{n-1}(sq^{-k-j}-s^{-1}q^{k+j}) \right)=0 \in \widetilde{\Li}_{\cN}.
\end{equation*}
This is equivalent to:
\begin{equation*}
 {n+j \brack j}_q \prod_{k=0}^{n-1}(sq^{-k-j}-s^{-1}q^{k+j}) \in \cC_{\cN}
 \end{equation*}
if $i,j \leq \cN-1 \text { and } j+n\geq \cN$, {which holds following the definition of the ideal $\cC_{\cN}$}.
\par
In turn, this shows that the coefficient of $v_i\otimes v_j$ in $R(v_i\otimes v_j)$ vanishes if  $j+n\geq \cN$. So, the image of $ R(v_i\otimes v_j)|_{\widetilde{\Li}}$  is a linear combination of tensors with two components such that their indices are all bounded by  $\cN-1$. This shows precisely that the action coming from the 2-variable $R$-matrix preserves the module $\widetilde{V}_{\cN}^{\otimes 2}$ (once specialised to the quotient $\widetilde{\Li}_{\cN}$) and concludes the proof.
\end{proof}
\begin{rmk}
Following the above proof we notice that %actually 
the subset $\left.\left<v_0, \cdots\!, v_{\cN-1}\right>\right|{\widetilde\Li_\cN}$ of $\hat V\mid_{\widetilde\Li_\cN}$ is also invariant under the action of the quantum group $U_q(sl_2)$.  
\end{rmk}
\subsection{Unique quantum trace at level $\cN$} 
We recall that our aim is to obtain knot invariants. In order to construct them our strategy will be to make use of the braid group action on tensor powers of the generic Verma module defined above. Our construction will be split into two main parts. {So far, we considered} a certain quotient ring such that we have a well defined braid action once we specialise the finite part from the Verma module over this ring. Now, once we have such an action, we are going to construct a unique quantum trace, which in turn will provide our knot invariants. 
\begin{defn}
A quantum partial trace is a map
$$
\qptr_{\widetilde{V}_{\cN}^{\otimes n}}:\operatorname{End}(\widetilde{V}_{\cN}^{\otimes (n+1)})\rightarrow \operatorname{End}(\widetilde{V}_{\cN}^{\otimes n})
$$
a map which has the form
\[
\qptr_{\widetilde{V}_{\cN}^{\otimes n}}(f)_{v_{i_1,i_2, \cdots, i_n}}^{v_{j_1,j_2, \cdots, j_n}}
:=\sum_{j=0}^{\cN-1}
g_j(q, s) \cdot f_{v_{i_1,i_2, \cdots, i_n,j}}^{v_{j_1,j_2, \cdots, j_n,j}}
\]
and which satisfies
\begin{equation}
\qptr_{\widetilde{V}_{\cN}^{\otimes n}}(\widetilde\rho_{\cN,{n+1}}(\sigma_n\beta_n)) 
=
a \widetilde\rho_{\cN,n}(\beta_n), \quad
\qptr_{\widetilde{V}_{\cN}^{\otimes n}}(\widetilde\rho_{\cN,{n+1}}(\sigma_n^{-1}\beta_n) )
=
a^{-1} \widetilde\rho_{\cN,n}(\beta_n)
\label{eq:qtcondition}
\end{equation}
 for {any }$\beta_n \in B_n$ and a non-zero element $a \in \widetilde\Li_\cN$.  
\end{defn}
\begin{thm}[Unique level $\cN$ quantum trace]\label{thm:uqtr1}
There exists a unique quantum trace $\qptr_{V_{\cN}^{\otimes n}}$, up to a scalar multiple, 
for the level $\cN$ sequence of braid group representations 
  $
 \{\widetilde\rho_{\cN,n}\mid n \in \N \}
 $.
 In other words, {for any choice of a quantum trace associated to $\widetilde{V}_{\cN}^{\otimes n}$} there is a non-zero element $g_0(q, s)\in \Li_\cN$ such that
\[
\qptr_{\widetilde{V}_{\cN}^{\otimes n}}(f)_{v_{i_1,i_2, \cdots, i_n}}^{v_{j_1,j_2, \cdots, j_n}}
:=\sum_{j=0}^{\cN-1}
g_0(q, s)q^{-2i} \cdot f_{v_{i_1,i_2, \cdots, i_n,j}}^{v_{j_1,j_2, \cdots, j_n,j}}.
\]
\end{thm}
\begin{proof}
First of all, it is well-known that if $g_j(q, s) = g_0(q, s)q^{-2i}$, then $g_j(q, s)$ satisfies the condition \eqref{eq:qtcondition}.

So we only have to prove that $g_j(q, s)$ must be $g_0(q, s)q^{-2i}$.  
Since the quantum partial trace of 
$\widetilde\rho_{\cN,n}(\sigma_n^{\pm1}\beta_n)$ 
is given by
\begin{multline*}
\qptr_{\widetilde{V}_{\cN}^{\otimes n}}\left(\widetilde\rho_{\cN,n}(\sigma_n^{\pm1}\beta_n)\right)_{v_{i_1,i_2, \cdots, i_n}}^{v_{j_1,j_2, \cdots, j_n}}
=
\\
=\sum_{j=0}^{\cN-1}g_j(s,q) \left(\widetilde\rho_{\cN,n}(\sigma_n^{\pm1}\beta_n)\right)_{v_{i_1,i_2, \cdots, i_n,j}}^{v_{i_1,i_2, \cdots, i_n,j}}
=
\\
=\sum_{j=0}^{\cN-1}g_j(s,q) \cdot 
\sum_{k=0}^{\cN-1}
(id^{\otimes n-1} \otimes R^{\pm1})_{v_{j_1, \cdots, j_{n-1},k,j}}^{v_{j_1, \cdots, j_{n-1},j_n,j}}
(\widetilde\rho_{\cN,n}(\beta_n)\otimes id)_{v_{i_1, \cdots,i_{n-1}, i_n,j}}^{v_{j_1, \cdots, j_{n-1},k,j}}
=
\\
=\sum_{j=0}^{\cN-1}g_j(s,q) \cdot 
\sum_{k=0}^{\cN-1}
(R^{\pm1})_{k,j}^{v_{j_n,j}}
\widetilde\rho_{\cN,n}(\beta_n)_{v_{i_1, \cdots,i_{n-1}, i_n}}^{v_{j_1, \cdots, j_{n-1},k}}
=
\\
=\sum_{j=0}^{\cN-1}g_j(s,q) \cdot 
(R^{\pm1})_{v_{j_n,j}}^{v_{j_n,j}}
\widetilde\rho_{\cN,n}(\beta_n)_{v_{i_1, \cdots,i_{n-1}, i_n}}^{v_{j_1, \cdots, j_{n-1},j_n}}
=
\\
=\left(\sum_{j=0}^{\cN-1}g_j(s,q) \cdot 
(R^{\pm1})_{v_{j_n,j}}^{v_{j_n,j}}
\right)
\rho_{\cN,n}(\beta_n)_{v_{i_1, \cdots,i_{n-1}, i_n}}^{v_{j_1, \cdots, j_{n-1},j_n}},
\end{multline*}
we only have to show that $\sum_{j=0}^{\cN-1}g_j(s,q) \cdot 
(R^{\pm1})_{v_{i,j}}^{v_{i,j}}$ is a non-zero scalar independent of $i$.  
\par
We see $\sum_{j=0}^{\cN-1}g_j(s,q) \cdot 
R_{v_{i,j}}^{v_{i,j}}$ to show that $g_j(s, q)$ must be $g_0(q,s) q^{-2j}$
by an induction on $M \in \{0,\dots,\cN\}$. 
To show
\begin{equation*}
g_i(q,s)=q^{-2i} \cdot g_0(q,s),\quad  \forall \  0\leq  i \leq M,
\end{equation*}
we use the fact
\begin{equation}
\sum_{j=0}^{M}q^{-2j} \cdot 
\left(R_{v_{M,j}}^{v_{M,j}}
-
R_{v_{M+1,j}}^{v_{M+1,j}}\right)
=
q^{-2(M+1)}R_{v_{M+1,M+1}}^{v_{M+1,M+1}}
\label{eq:sum}
\end{equation}
given in the following Lemma \ref{lem:qptr}.  
Now we prove that, if the sum  $\sum_{j=0}^{\cN-1}g_j(s,q) \cdot 
R_{v_{i,j}}^{v_{i,j}}$ is not depend on $i$, then
\[
\text{P(M):} \qquad g_j(q, s) = g_0(q, s) q^{-2j}\quad 
\text{for $0 \leq j \leq M$}
\]
holds by induction on $M$.  
 \par
For the case $M=0$, we have 
\begin{multline*}
0 = g_0(q,s) \cdot 
R_{v_{0,0}}^{v_{0,0}} - g_0(q,s) \cdot 
R_{v_{0,1}}^{v_{0,1}}- g_1(q,s) \cdot 
R_{v_{1,1}}^{v_{1,1}}
=
(g_0(q,s) q^{-2} - g_1(q, s)) R_{v_{1,1}}^{v_{1,1}}
\end{multline*}
by \eqref{eq:sum} and we get $g_1(q,s) = q^{-2}$.  
\par
Assuming that P(M) is true, then 
  \begin{multline*}
0 = \sum_{j=0}^{M}g_j(q, s) \cdot 
R_{v_{M,j}}^{v_{M,j}}
-
\sum_{j=0}^{M+1}g_j(q, s) \cdot 
R_{v_{M+1,j}}^{v_{M+1,j}}
=
\\
\sum_{j=0}^{M}g_0(q, s)q^{-2j} \cdot 
R_{v_{M,j}}^{v_{M,j}}
-
\sum_{j=0}^{M}g_0(q, s)q^{-2j} \cdot 
R_{v_{M+1,j}}^{v_{M+1,j}}
- g_{M+1}(q, s) R_{v_{M+1,M+1}}^{v_{M+1,M+1}}
=
\\
(g_0(q,s)q^{-2(M+1)} 
- g_{M+1}(q, s) ) R_{v_{M+1,M+1}}^{v_{M+1,M+1}}.
\end{multline*}
Hence $g_{M+1}(q, s) = q_0(q, s) q^{-2(M+2)}$ and P(M+1) is true.  
\par
We have
\[\sum_{j=0}^{0}g_0(q, s)q^{-2j} \cdot 
R_{v_{0,j}}^{v_{0,j}}=
g_0(q,s)R_{v_{0,0}}^{v_{0,0}}
= g_0(q,s),
\]
and
\[
\sum_{j=0}^{\cN-1}g_0(q, s)q^{-2j} \cdot 
(R^{-1})_{v_{\cN-1,j}}^{v_{\cN-1,j}}
=
g_0(q, s){q^{-2(\cN-1)}}(R^{-1})_{v_{\cN-1,\cN-1}}^{v_{\cN-1,\cN-1}}
=
s^{2\cN-2}q^{{-2\cN(\cN-1)}}. 
\]
Therefore, by putting $g_{0}(q, s) = s^{-\cN+1}q^{{\cN(\cN-1)}}$,
we get \eqref{eq:qtcondition} with $a = s^{-\cN+1}q^{{\cN(\cN-1)}}$.  
\end{proof}
\begin{rmk}
The matrix element $R_{v_{M+1,M+1}}^{v_{M+1,M+1}}$ is  invertible  in $\Li$, so we also get $g_{M+1}(q, s) = q_0(q, s) q^{-2(M+2)}$ in $\widetilde \Li_\cN$ and $\Li_\cN$.
\end{rmk}
\begin{lem}\label{lem:qptr}
The following holds for any $M \in \{0,\dots,\cN-2\}$:
\begin{equation}
\sum_{j=0}^{M}q^{-2j} \cdot 
\left(R_{v_{M,j}}^{v_{M,j}}
-
R_{v_{M+1,j}}^{v_{M+1,j}}\right)
=
q^{-2(M+1)}R_{v_{M+1,M+1}}^{v_{M+1,M+1}}.
\label{eq:sum1}
\end{equation}
\end{lem}
\begin{proof}
Recall that
\[
R_{v_{i,j}}^{v_{i,j}}
=
s^{-(i+j)}q^{\frac{(i+j)^2-i+j}{2}}
{i \brack j}_q \prod_{k=0}^{i-j-1}(sq^{-k-j}-s^{-1}q^{k+j}).
\]
Hence we have
\begin{multline*}
R_{v_{M,0}}^{v_{M,0}}
-
R_{v_{M+1,0}}^{v_{M+1,0}}
=
s^{-M}q^{\frac{M^2-M}{2}}
{M \brack 0}_q \prod_{k=0}^{M-1}(sq^{-k}-s^{-1}q^{k})
-
\\
s^{-(M+1)}q^{\frac{M^2-M+2M}{2}}
{M+1 \brack 0}_q \prod_{k=0}^{M}(sq^{-k}-s^{-1}q^{k})
\\
=
s^{-M}q^{\frac{M^2-M}{2}}  \left(
1-1+s^{-2}q^{2M}
\right)
 \prod_{k=0}^{M-1}(sq^{-k}-s^{-1}q^{k})
\\
=
s^{-M-2}q^{\frac{M^2+3M}{2}} 
 \prod_{k=0}^{M-1}(sq^{-k}-s^{-1}q^{k}),
\end{multline*}
\begin{multline*}
R_{v_{M,0}}^{v_{M,0}}
-
R_{v_{M+1,0}}^{v_{M+1,0}}
-
q^{-2} R_{v_{M+1,1}}^{v_{M+1,1}}
=
\\
s^{-M-2}q^{\frac{M^2+3M}{2}} 
 \prod_{k=0}^{M-1}(sq^{-k}-s^{-1}q^{k})
 -
 \\
 s^{-M-2}q^{\frac{M^2 + 3M}{2}}
 {M+1 \brack 1}_q 
  \prod_{k=0}^{M-1}(sq^{-k-1}-s^{-1}q^{k+1})
=
\\
s^{-M-2}q^{\frac{M^2+3M}{2}} 
\left(
(s-s^{-1})
 -
\frac{\{M+1\}}{\{1\}}
 (sq^{-M}-s^{-1}q^{M})\right)
 \prod_{k=1}^{M-1}(sq^{-k}-s^{-1}q^{k})
=
\\
s^{-M-2}q^{\frac{M^2+3M}{2}} 
\times\hfill
\\
\frac{1}{\{1\}} 
\left((q-q^{-1})
(s-s^{-1})
 -
(q^{M+1} - q^{-M-1})
 (sq^{-M}-s^{-1}q^{M})
 \right)
 \prod_{k=1}^{M-1}(sq^{-k}-s^{-1}q^{k})
=
\\
-s^{-M-2}q^{\frac{M^2+3M}{2}} 
 {M \brack 1}_q 
\left(sq^{-M-1} - 
s^{-1}q^{M+1}
 \right)
 \prod_{k=0}^{M-2}(sq^{-k-1}-s^{-1}q^{k+1}), 
\end{multline*}
\begin{multline*}
R_{v_{M,0}}^{v_{M,0}}
-
R_{v_{M+1,0}}^{v_{M+1,0}}
-
q^{-2} (R_{v_{M+1,1}}^{v_{M+1,1}}
-
R_{v_{M,1}}^{v_{M,1}})
=
\\
-s^{-M-2}q^{\frac{M^2+3M}{2}} 
 {M \brack 1}_q 
\left(sq^{-M-1} - 
s^{-1}q^{M+1}
 \right)
 \prod_{k=0}^{M-2}(sq^{-k-1}-s^{-1}q^{k+1})
 +
 \\
s^{-M-1}q^{\frac{M^2+M-2}{2}}
 {M \brack 1}_q 
  \prod_{k=0}^{M-2}(sq^{-k-1}-s^{-1}q^{k+1})  
=
\\
-s^{-M-2}q^{\frac{M^2+M-2}{2}} 
 {M \brack 1}_q 
\left(q^{M+1}(sq^{-M-1} - 
s^{-1}q^{M+1}
 )
 -
 s
\right)
\prod_{k=0}^{M-2}(sq^{-k-1}-s^{-1}q^{k+1})
=
\\
s^{-M-3}q^{\frac{M^2+5M+2}{2}} 
 {M \brack 1}_q 
 \prod_{k=0}^{M-2}(sq^{-k-1}-s^{-1}q^{k+1}),
 \end{multline*}
 \begin{multline*}
R_{v_{M,0}}^{v_{M,0}}
-
R_{v_{M+1,0}}^{v_{M+1,0}}
-
q^{-2} (R_{v_{M+1,1}}^{v_{M+1,1}}
-
R_{v_{M,1}}^{v_{M,1}})
-
q^{-4} (R_{v_{M+1,2}}^{v_{M+1,2}}
-
R_{v_{M,2}}^{v_{M,2}})
=
\\
s^{-M-4}q^{\frac{M^2+7M+6}{2}} 
 {M \brack 2}_q 
 \prod_{k=0}^{M-3}(sq^{-k-1}-s^{-1}q^{k+1}),
 \end{multline*}
 \begin{multline*}
R_{v_{M,0}}^{v_{M,0}}
-
R_{v_{M+1,0}}^{v_{M+1,0}}
-
q^{-2} (R_{v_{M+1,1}}^{v_{M+1,1}}
-
R_{v_{M,1}}^{v_{M,1}})
-
q^{-4} (R_{v_{M+1,2}}^{v_{M+1,2}}
-
R_{v_{M,2}}^{v_{M,2}})
-
q^{-6} (R_{v_{M+1,3}}^{v_{M+1,3}}
-
R_{v_{M,3}}^{v_{M,3}})
=
\\
s^{-M-5}q^{\frac{M^2+9M+12}{2}} 
 {M \brack 3}_q 
 \prod_{k=0}^{M-4}(sq^{-k-1}-s^{-1}q^{k+1}).
 \end{multline*}
Following an inductive argument, we deduce this property for any natural number $M$, $M \leq \cN-2$. Then we get
\begin{equation*}
\sum_{j=0}^M
q^{-2j}\left(
R_{v_{M,j}}^{v_{M,j}}
-
R_{v_{M+1,j}}^{v_{M+1,j}}
\right)
=
s^{-2M-2}q^{\frac{4M^2+4M}{2}} 
=
q^{-2M-2}
R_{v_{M+1,M+1}}^{v_{M+1,M+1}}.
\end{equation*}
\end{proof}
\subsection{Level $\cN$ unified invariant}
For the next part of our construction we want to use the uniquely defined quantum traces on these quotient modules together with the induced braid group actions that come from the generic R-matrix in order to obtain knot invariants. 
We recall that we are working over the following quotient ring defined in \eqref{qring}:
\begin{equation*}
\widetilde\Li_{\cN}=\Q[q^{\pm2}, s^{\pm 2}] \ / \widetilde\cI_\cN.
\end{equation*}
\begin{notation}(Partial trace associated to a vector)
We extend  the  quantum partial trace $\qptr$ on $\widetilde{V}_{\cN}^{\otimes 2}$ to $\widetilde{V}_{\cN}^{\otimes n}$, which is given by:
\[
\qt_{\widetilde{V}_{\cN}}:=\qptr_{\widetilde{V}_{\cN}} \circ \qptr_{\widetilde{V}_{\cN}^{\otimes 2}} \circ \cdots \circ \qptr_{\widetilde{V}_{\cN}^{\otimes(n-1)}} : \operatorname{End}(\widetilde{V}_{\cN}^{\otimes n})\rightarrow \operatorname{End}(\widetilde{V}_{\cN}).
\]
where its evaluation on the vector $v_i$ as:
\begin{equation*}
\qt_{\widetilde{V}_{\cN}}(f)_{v_i}^{v_i}=s^{-(\cN-1)(n-1)}q^{{\cN(\cN-1)}(n-1)}
\sum_{j_2, j_3 \cdots, j_n=0}^{\cN-1}
q^{-2\sum_{k=2}^n j_k}  f_{v_{i, j_2, \cdots  j_n}}^{v_{i, j_2, \cdots  j_n}}.
\end{equation*}
\end{notation}
\begin{defn}
\label{O}
For $\cN \in \mathbb N$ and for $\beta_n \in B_n$ let us consider 
$
\widetilde\Omega(\beta_n)(q,s) \in \Li_{\cN}
$
 obtained from the above braid action via the quantum partial trace on the first component:
\begin{equation*}
\widetilde\Omega_{\cN}(\beta_n)(q,s)
= 
s^{w(\beta_n)(\cN-1)}q^{-w(\beta_n){\cN(\cN-1)}} \cdot 
\qt_{\widetilde{V}_{\cN}}\left( \widetilde\rho_{\cN,n}(\beta_n)
\right)_{v_0}^{v_0} \in \Li_{\cN}.
\end{equation*}
\end{defn}
The properties of the quantum partial trace implies Theorem \ref{thm:THEOREM} and Corollary \ref{coro:prime} which we remind below.
\begin{thm}[{\bf Level $\cN$ maximal universal invariant}]\label{thm:THEOREMG}
For any $\cN \in \N$, $\cN \geq 2$, $\widetilde{\Li}_{\cN}$ is the largest quotient of the polynomial ring $\Liq$ with the property that the image of the quantum trace in this quotient gives a knot invariant. 
\par
More precisely, the image of the level $\cN$ quantum trace in this quotient ring $\widetilde{\Omega}_{\cN}(L)(q,s) \in \widetilde{\Li}_{\cN}$ is a knot invariant. 
Also, if $\widetilde{\Li}_{\cN}'$ is a quotient of $ \Li$ in which $\widetilde{\Omega}_{\cN}(\beta_n)$ is a knot invariant, then the quotient onto $\widetilde{\Li}_{\cN}'$ factors through $\widetilde{\Li}_{\cN}$.
\end{thm}
\begin{coro}[{\bf Level $\cN$ maximal unified invariant at prime parameters}]
\label{coro:COROLLARY}
If $\cN \in \mathbb N$ is prime, the ring $\Liq_{\cN}$ is the largest quotient of the polynomial ring $\Li$ with the property that the image of the quantum trace in this quotient gives a knot invariant. 
More precisely, for any $\cN \in \cN$, $\cN\geq 2$ we have that 
\[
\Omega_{\cN}(L)(q,s) = \pi_{\cN}(\widetilde{\Omega}_{\cN}(\beta_n)) \in \Li_{\cN}
\]
 is a well defined oriented knot invariant where $\pi_\cN$ is the natural projection from $\widetilde\Li_\cN$ to $\Li_\cN$. 
 Also, if $\cN$ prime and we consider $\Li_{\cN}'$ a quotient of $ \Li$ in which $\Omega_{\cN}(\beta_n)$ is a knot invariant, then the quotient onto $\Li_{\cN}'$ factors through $\Li_{\cN}$.
\end{coro}
\begin{rmk}
[Level $\cN$ invariant over the integers]
\label{ringz} 
We would like to remark that all the construction works over the quotient ring of $\Z[q^{\pm2}, s^{\pm 2}]$ instead of $\Q[q^{\pm2}, s^{\pm 2}]$, quotiented by the ideal generated by the same elements. 
Let $\Li^{\Z}_{\cN} = \Z[q^{\pm2}, s^{\pm 2}]/\cI_\cN$, then we obtain a well-defined knot invariant
$\Omega_{\cN}(L)(q,s) \in \Li^{\Z}_{\cN}$.
\end{rmk}
 \section{Interpolation formula in the level $\cN$-quotient ring}
 \label{rec}
  In this section we are going to prove Theorem \ref{TINT}, which provides a formula for the level $\cN$ interpolation invariant. In order to prove this, we will first show that both our level $\cN$ interpolation and level $\cN$ universal invariants recover the coloured Jones and ADO invariants.
 \subsection{Recovering coloured Jones and ADO invariants} 
 \begin{thm}[The level $\cN$ unified invariant recovers coloured Jones and ADO invariants]\label{JA1} We have that: 
 \begin{equation*}
\Omega_{\cN}(L)\mid_{s = q^{1-\mathcal{N}}} = J_{\mathcal{N}}(L, q), \qquad
\Omega_{\cN}(L)\mid_{q=\xi_{\cN}}=\Phi_{\mathcal{N}}(L, s). 
\end{equation*}
 \end{thm}
\begin{proof}
We recall that the invariant $\Omega_{\cN}(L)(q,s)$ is constructed from the representation theory of $U_q(sl_2)$ over the ring $\Li$. It uses the induced braid action on the specialisation of the $\cN$-part from the Verma module:
$$
V_{\cN}:=\langle v_{0},...,v_{\cN-1}\rangle_{\Li_{\cN}} \subseteq \hat{V}\mid_{\Li_{\cN}}.
$$
Further on, if we specialise again at $s = q^{1-\mathcal{N}}$ , we obtain that $V_{\cN}\mid _{s = q^{1-\mathcal{N}}}$ is the usual $\cN$-dimensional representation of $U_q(sl_2)$ at generic $q$. This shows that :
$$
\Omega_{\cN}(L)\mid_{s = q^{1-\mathcal{N}}} = J_{\mathcal{N}}(L, q).
$$
Dually, if we consider the specialisation at roots of unity $q=\xi_{\cN}$ we have that $V_{\cN}\mid_{q=\xi_{\cN}}$ is the is the $\cN$-dimensional representation of $U_q(sl_2)$ at generic the $2\cN^{th}$ root of unity. This shows that we also have:
$$
\Omega_{\cN}(L)\mid_{q-\xi_{\cN}} = \Phi_{\mathcal{N}}(L, s)
$$
which concludes the proof of the statement.
\end{proof}
There is a natural projection $\pi_\cN : \widetilde\Li_\cN \to \Li_\cN$, we have the following.
  \begin{coro}[The level $\cN$ universal invariant recovers coloured Jones and ADO invariants] We have that: \begin{equation*}
 \widetilde{\Omega}_{\cN}(L)\mid_{s = q^{1-\mathcal{N}}} 
 =
  J_{\mathcal{N}}(L, q), \qquad
\widetilde{\Omega}_{\cN}(L)\mid_{q=\xi_{\cN}}=\Phi_{\mathcal{N}}(L, s). 
\end{equation*}
 \end{coro}
 \subsection{Interpolation formula}
 Now we are ready to prove Theorem \ref{TINT}, which recall below.
 \begin{thm}[The level $\cN$ intersection form interpolates coloured Jones and ADO invariants] 
\label{JA3} 
We have that  $\Omega_{\cN}(L)(q,s) \in \Li_{\cN}$ is a well defined oriented knot invariant and globalises the $\cN^{th}$ coloured Jones and coloured Alexander invariants, as follows: 
\begin{equation*}
\Omega_{\cN}(L)(q,s) = J_{\mathcal{N}}(L, q) + \Phi_{\mathcal{N}}({L}, s) - \Phi_{\mathcal{N}}(L, q^{1-\mathcal N}).
\end{equation*}
 \end{thm}
\begin{proof}
The formula for our quotient ring is the following, as defined in relation \eqref{qring}:
\begin{equation*}
\Li_{\cN}=\Q[q^{\pm2}, s^{\pm 2}] \ / \left(\varphi_\cN(q^2)\cdot (s^2q^{-2(\cN-1)}-1)\right)
\subseteq \Liq=\Q[q^{\pm2}, s^{\pm 2}].
\end{equation*}
We start with the specialization of our intersection form at roots of unity, which we know that leads to the $\cN^{th}$ coloured Alexander polynomial:
$$\Omega_{\mathcal{N}}|_{q = \xi_{\mathcal{N}}} = \Phi_{\mathcal{N}}(L)(s).$$
 So there exists $\Omega_{\mathcal{N}}^\Phi(L)(s, q) \in \Li_{\cN}$ such that 
\begin{equation*}
\Omega_{\mathcal{N}}(L)(s, q) - \Phi_{\mathcal{N}}(L)(s) = 
\Omega_{\mathcal{N}}^\Phi(L)(s, q)\, \varphi_{\mathcal{N}}(q^2).
\end{equation*}
We recall that $\varphi_{\mathcal{N}}(q)$ is the $\cN^{th}$ cyclotomic polynomial.
This means that:
\begin{equation}\label{r1}
\Omega_{\mathcal{N}}(L)(s, q)= \Phi_{\mathcal{N}}(L)(s) + 
\Omega_{\mathcal{N}}^\Phi(L)(s,q)\, \varphi_{\mathcal{N}}(q^2).
\end{equation}
On the other hand, the specialisation at natural parameters gives the $\cN^{th}$ coloured Jones polynomial:
$$
\Omega_{\mathcal{N}}(L)(q^{1-\mathcal{N}},\, q) = J_{\mathcal{N}}(L)(q).
$$
This shows that the following relation holds:
\begin{equation}\label{r2}
 \Phi_{\mathcal{N}}(L)(q^{1-\cN}) + 
\Omega_{\mathcal{N}}^\Phi(L)(q^{1-\cN},q)\, \varphi_{\mathcal{N}}(q^2)=J_{\mathcal{N}}(L)(q).
\end{equation}
On the other hand, by the definition of quotient rings we see that there exists $\Omega_{\mathcal{N}}^J(L)(s, q) \in \Li_{\cN}$ such that:
\begin{equation*}%\label{r3}
\Omega_{\mathcal{N}}^\Phi(L)(s,q)=\Omega_{\mathcal{N}}^\Phi(L)(q^{1-\cN},q)+\Omega_{\mathcal{N}}^J(L)(s, q) \cdot(s-1^{1-\cN}).
\end{equation*}
Then, in the quotient ring we have:
\begin{equation*}
\Omega_{\mathcal{N}}^\Phi(L)(s,q) \cdot \varphi_{\mathcal{N}}(q^2)=\Omega_{\mathcal{N}}^\Phi(L)(q^{1-\cN},q) \cdot \varphi_{\mathcal{N}}(q)+ \Omega_{\mathcal{N}}^J(L)(s, q).
\end{equation*}
since $\varphi_{\mathcal{N}}(q^2)(s-q^{1-\cN})$ becomes zero in $\Li_{\cN}$.
Combining this property with equation \eqref{r2} we obtain:
\begin{equation*}
\Omega_{\mathcal{N}}^\Phi(L)(s,q)\, \varphi_{\mathcal{N}}(q^2)=J_{\mathcal{N}}(L)(q)- \Phi_{\mathcal{N}}(L)(q^{1-\cN}).
\end{equation*}
This relation together with equation \eqref{r1} show that:
\begin{equation*}
\Omega_{\mathcal{N}}(L)(s, q)
=J_{\mathcal{N}}(L)(q) +\Phi_{\mathcal{N}}(L)(s)- \Phi_{\mathcal{N}}(L)(q^{1-\mathcal{N}}).
\end{equation*}
This concludes the interpolation formula presented in the statement. 
\end{proof}
\section{Appendix}\label{Ap}
\setcounter{section}{1}
\renewcommand{\thesection}{\Alph{section}}
\setcounter{equation}{0}
\setcounter{defn}{0}
Here we give the proof of Poposition \ref{prop:EQ}.
We first prove $\cC_\cN = \widetilde{\cJ}_{\cN}$ and then prove  $\widetilde\cJ_\cN = \widetilde\cI_\cN$.  
\begin{lem}
\label{lem:EQ3}
The   three  structural ideals introduced in Definitions \ref{def:condition} and \ref{def:maximal} are equal, i.e. 
\begin{equation*}
\cC_\cN = \widetilde{\cJ}_{\cN}. 
\end{equation*}
\end{lem}

\begin{proof}
We use the following intermediate ideal 
\begin{equation*}%\label{eq:INprime}
\cI'_\cN = \sum_{k=0}^{\cN-2} \left( {\cN \brack k+1}_q \cdot \prod_{i=1}^{\cN-k-1}(s^2q^{-2k-2i}-1) \right)\subseteq \Liq.
\end{equation*}
We first show that we have the equality $\cI_{\cN}^\prime = \widetilde{\cJ}_{\cN}$, in two parts, as below.
\par
{\bf Step I: $\cI_{\cN}^\prime \subset \widetilde{\cI}_{\cN}$} 
\par  
Each ideal $\left( {\cN \brack k+1}_q \cdot \prod_{i=1}^{\cN-k-1}(s^2q^{-2k-2i}-1) \right)$ in 
$\cI_\cN^\prime$ contains the factor $(s^2 q^{-2\cN+2}-1)$, so 
\begin{equation}\label{r1'}
\cI^\prime_\cN = (s^2 q^{-2\cN+2}-1)\left(\sum_{k=0}^{\cN-2}
 \Big({\cN \brack k+1}_q \cdot \prod_{i=1}^{\cN-k-2}(s^2q^{-2k-2i}-1) \Big)\right).
\end{equation} 
Now we look at the ideal $\sum_{k=0}^{\cN-2}
 \Big({\cN \brack k+1}_q \cdot \prod_{i=1}^{\cN-k-2}(s^2q^{-2k-2i}-1) \Big)$. 
 \begin{equation}\label{r2'}
 \begin{aligned} \sum_{k=0}^{\cN-2}
 \Big({\cN \brack k+1}_q \cdot \prod_{i=1}^{\cN-k-2}(s^2q^{-2k-2i}-1) \Big)
 &=
\sum_{k=1}^{\cN-1}
  \Big({\cN \brack k}_q \cdot \prod_{i=1}^{\cN-k-1}(s^2q^{-2k-2i+2}-1) \Big)
 =\\
&=\sum_{k=1}^{\cN-1}
  \Big({\cN \brack k}_q \cdot \prod_{i=0}^{\cN-k-2}(s^2q^{-2k-2i}-1) \Big).
 \end{aligned}
 \end{equation}
 Now, for the term associated to $k=\cN-2$ we write the quantum number $[\cN]_{q}$ as a product of all $\varphi_d(q^2)$ for $d\mid \cN, d \neq 1$, and so we obtain the following description of the ideals:
 \begin{equation*}
 \begin{aligned}
&\Big(\bigcap_{\text{\scriptsize$\begin{matrix}
 d\mid \cN \\ d \neq 1 \end{matrix}$}} \Big( \varphi_d(q^2)\Big)\Big) +
\sum_{k=1}^{\cN-2}
  \Big({\cN \brack k}_q \cdot \prod_{i=0}^{\cN-k-2}(s^2q^{-2k-2i}-1) \Big) 
  =\\
 &=\bigcap_{\text{\scriptsize$\begin{matrix}
 d\mid \cN \\ d \neq 1 \end{matrix}$}} \left(\Big( \varphi_d(q^2)\Big) +
\sum_{k=1}^{\cN-2}
  \Big({\cN \brack k}_q \cdot \prod_{i=k}^{\cN-2}(s^2q^{-2i}-1) \Big)  \right).
 \end{aligned}
 \end{equation*}
 The cyclotomic plynomial $\varphi_d(q^2)$ divide ${\cN \brack k}_q$ if and only if $k$ is not a multiple of $d$.  
 Moreover, $\varphi_d(q^2)$ and ${\cN \brack ld}_q$ are mutually prime.  
 We also know that $q^{2d}$ is equal to $1$ modulo $\varphi_d(q^2)$.  
 Therefore, we have
\begin{equation*}
\Big( \varphi_d(q^2)\Big) +
\sum_{k=1}^{\cN-2}
  \Big({\cN \brack k}_q \cdot \prod_{i=k}^{\cN-2}(s^2q^{-2i}-1) \Big)  
=
\Big( \varphi_d(q^2)\Big) +
\sum_{l=1}^{\cN/d-1}
  \Big({\cN \brack ld}_q \cdot \prod_{i=ld}^{\cN-2}(s^2q^{-2i}-1) \Big). 
\end{equation*}
Replacing $l$ by $\cN/d - l$, we get
\begin{equation*}
\Big( \varphi_d(q^2)\Big) +
\sum_{l=1}^{\cN/d-1}
  \Big({\cN \brack ld}_q \cdot \prod_{i=ld}^{\cN-2}(s^2q^{-2i}-1) \Big)
  =
\Big( \varphi_d(q^2)\Big) +
\sum_{l=1}^{\cN/d-1}
  \Big({\cN \brack ld}_q\prod_{i=0}^{ld-2}(s^2q^{-2i}-1) \Big) ,
\end{equation*}
which is contained in $\left(\varphi_d(q^2), \  \ \prod_{i=0}^{d-2}(s^2q^{-2i}-1) \right)$. 
Hence, $\cI_\cN^\prime$ is contained in $\widetilde{\cJ}_\cN$.
\par
{\bf Step II: $\widetilde{\cJ}_\cN \subseteq \cI_\cN^\prime$}
\par
By Lemma \ref{lem:qi} given below, we have ${\cN \brack ld}_q = \pm{\cN/d \brack l}_1$, which is a non-zero integer and is invertible in $\mathbb{Q}$, so we obtain:
\begin{equation*}
\begin{aligned}
\Big( \varphi_d(q^2)\Big) +
\sum_{l=1}^{\cN/d-1}
  \Big({\cN \brack ld}_q\prod_{i=0}^{ld-2}(s^2q^{-2i}-1) \Big) 
    &=
 \Big( \varphi_d(q^2)\Big) +
\sum_{l=1}^{\cN/d-1}
  \Big(\prod_{i=0}^{ld-2}(s^2q^{-2i}-1) \Big) 
  =\\
  &=
   \Big( \varphi_d(q^2)\Big) +
  \Big(\prod_{i=0}^{d-2}(s^2q^{-2i}-1) \Big).
 \end{aligned}
\end{equation*} 
Then, following relations \eqref{r1'} and \eqref{r2'} we have:
\begin{equation*}
\begin{aligned}
\cI^\prime_\cN &= (s^2 q^{-2\cN+2}-1)\left(\sum_{k=0}^{\cN-2}
 \Big({\cN \brack k+1}_q \cdot \prod_{i=1}^{\cN-k-2}(s^2q^{-2k-2i}-1) \Big)\right)=\\
 &=(s^2 q^{-2\cN+2}-1) \cdot \bigcap_{\text{\scriptsize$\begin{matrix}
 d\mid \cN \\ d \neq 1 \end{matrix}$}} \Big(\Big( \varphi_d(q^2)\Big) +
  \Big(\prod_{i=0}^{d-2}(s^2q^{-2i}-1) \Big) \Big).
 \end{aligned}
\end{equation*} 
On the other hand, let us recall that 
\begin{equation*}
\widetilde{\cJ}_\cN =   
\left((s^2 q^{-2\cN+2} - 1) \cdot\left(\varphi_\cN(q^2)\right)\right) \cap
 \bigcap_{\text{\scriptsize$\begin{matrix}
 d\mid \cN \\ d \neq 1, \cN \end{matrix}$}}
 \left(\varphi_d(q^2),\ \  \prod_{i=0}^{d-2}(s^2q^{-2i}-1) \right)
\end{equation*}
The previous relations imply that $\widetilde{\cI}_\cN \subseteq \cI_\cN^\prime$ and so we conclude the equality $\cI_\cN^\prime = \widetilde{\cI}_\cN$.  
\par
Next, we will show that $\cI_{\cN}^\prime = \cC_{\cN}$ in two steps.
\par
{\bf Step III: $\cI_{\cN}^\prime \subset \cC_{\cN}$ } 
\par
For this, we start with the definition of $\cC_\cN$:
\begin{multline*}
\cC_\cN= \sum_{\substack{i,j=0 \\ j+n\geq \cN,0\leq  n\leq i}}^{\cN-1}
 \left( {n+j \brack j}_q \prod_{k=0}^{n-1}(sq^{-k-j}-s^{-1}q^{k+j}) \right) 
 \\
 =
 \sum_{i=0}^{\cN-1}
 \sum_{j=1}^{\cN-1}
 \sum_{n=\cN-j}^{i}
 \left( {n+j \brack j}_q \prod_{k=0}^{n-1}(sq^{-k-j}-s^{-1}q^{k+j}) \right) 
 \\=
 \sum_{i=0}^{\cN-1}
 \sum_{j=0}^{\cN-2}
 \sum_{n=\cN-j-1}^{i}
 \left( {n+j+1 \brack j+1}_q \prod_{k=0}^{n-1}(sq^{-k-j-1}-s^{-1}q^{k+j+1}) \right).  
\end{multline*}
For $n=\cN-j-1$, 
\begin{multline*}
 \left( {n+j+1 \brack j+1}_q \prod_{k=0}^{n-1}(sq^{-k-j-1}-s^{-1}q^{k+j+1}) \right)  
 =
 \\
  \left( {\cN \brack j+1}_q \prod_{k=0}^{\cN-j-2}(sq^{-k-j-1}-s^{-1}q^{k+j+1}) \right)  
  =
  \\
   \left( {\cN \brack j+1}_q \prod_{k=1}^{\cN-j-1}(sq^{-k-j}-s^{-1}q^{k+j}) \right).
\end{multline*}
Therefore, we obtain that $\cI_{\cN}^\prime \subset \cC_{\cN}$.  
\par
{\bf Step IV: $\cC_{\cN}\subset \cI_{\cN}^\prime$} 
\par
For this last step we are going to show that:
\begin{equation*}
\left( {n+j+1 \brack j+1}_q \prod_{k=0}^{n-1}(sq^{-k-j-1}-s^{-1}q^{k+j+1}) \right) 
\in \cI_\cN^\prime.
\end{equation*}
Let $l = n+j+1 - \cN$.  
Then
\begin{equation*}
\left( {n+j+1 \brack j+1}_q \prod_{k=0}^{n-1}(sq^{-k-j-1}-s^{-1}q^{k+j+1}) \right)
=
\left( {\cN+l \brack j+1}_q \prod_{k=1}^{\cN+l-j-1}(sq^{-k-j}-s^{-1}q^{k+j}) \right).
\end{equation*}
We show that the above is contained in $\cI_\cN^\prime$ by an induction on $l$.  
If $l=0$, then it is contained in $\cI_\cN^\prime$.  
Now we assume that $\left( {\cN+l \brack j+1}_q \prod_{k=1}^{\cN+l-j-1}(sq^{-k-j}-s^{-1}q^{k+j}) \right) \in \cI_\cN^\prime$.  
Then
\begin{multline*}
\left( {\cN+l+1 \brack j+1}_q \prod_{k=1}^{\cN+l-j}(sq^{-k-j}-s^{-1}q^{k+j}) \right)
=
\\
q^{\cN+l-j}\left( {\cN+l \brack j+1}_q \prod_{k=1}^{\cN+l-j}(sq^{-k-j}-s^{-1}q^{k+j}) \right)
+
\\
q^{-j-1}\left( {\cN+l \brack j}_q \prod_{k=1}^{\cN+l-j}(sq^{-k-j}-s^{-1}q^{k+j}) \right).  
\end{multline*}
The last  two terms are contained in $\cI_\cN^\prime$ and so
$\left( {\cN+l+1 \brack j+1}_q \prod_{k=1}^{\cN+l-j}(sq^{-k-j}-s^{-1}q^{k+j}) \right) \in \cI_\cN^\prime$.  
\end{proof}
\begin{lem}\label{lem:qi}
If $k$ and $l$ are both multiple of $d$, then ${k \brack l}_q$
is equal to ${k/d \brack l/d}_1$ or $-{k/d \brack l/d}_1$ modulo $\Phi_d(q^2)$.  
\end{lem}
\begin{proof}
First of all, we have
\[
[d + k] = -[k] \mod \Phi_d(q^2).  
\]
This implies that
\begin{equation}\label{eq:quantum integerd}
\frac{[ld + d-1][ld+d-2]\cdots[ld+1]}
{[d-1][d-2]\cdots[1]} = (-1)^{l(d-1)}.  
\end{equation}
We also know that
\begin{equation}\label{eq:quantumintegerquotient}
\frac{[kd]}{[d]} = 
q^{(k-1)d} + q^{(k-3)d} + \cdots + q^{-(k-1)d}
= \pm k \mod \Phi_d(q^2).  
\end{equation}
Let $k = k'd$ and $l = l'd$, then, by using \eqref{eq:quantum integerd} and \eqref{eq:quantumintegerquotient}, we get
\begin{equation*}
{k'd \brack l'd}_q 
=
\pm \frac{[k'd][(k'-1)d]\cdots [(k'-l'+1)d]}
{[l'd][(l'-1)d]\cdots[d]}
=
\pm {k' \brack l'}_1 \mod \Phi_d(q^2).  
\end{equation*}
Hence ${k'd \brack l'd}_q$ is equal to $\pm {k' \brack l'}_1$ modulo $\Phi_d(q^2)$.  
\end{proof}
To show that $\widetilde{\cJ}_\cN$ and $\widetilde\cI_{\cN}$ are equal, we use the following lemma. 
\begin{lem}[Structure of the maximal ideal at non-prime parameters] \label{strcy}The ideal $\widetilde{\cJ}_\cN$ has the following structure: 
\begin{equation*}
\widetilde{\cJ}_\cN
=\left( \varphi_\cN(q^2) \cdot (s^2 q^{-2\cN+2} - 1)\right)  \cap \bigcap_{\text{\scriptsize$\begin{matrix}
 d\mid \cN \\ d \neq 1, \cN \end{matrix}$}}
 \left(\varphi_d(q^2), \frac{s^{2d}-1}{s^2-1} \right).
\end{equation*}
 \end{lem}
\begin{proof}
We recall that
\begin{equation*}
\widetilde{\cJ}_\cN
=\left( \varphi_\cN(q^2) \cdot (s^2 q^{-2\cN+2} - 1)\right)  \cap \bigcap_{\text{\scriptsize$\begin{matrix}
 d\mid \cN \\ d \neq 1, \cN \end{matrix}$}}
 \left(\varphi_d(q^2), \ \prod_{i=1}^{d-1}(s^2q^{-2i}-1)  \right).
\end{equation*}
In the next part, we are going to prove that:
\begin{equation*}
\left(\varphi_d(q^2),\ \  \prod_{i=1}^{d-1}(s^2q^{-2i}-1) \right)=\left(\varphi_d(q^2),\ \  \frac{s^{2d}-1}{s^2-1} \right).
\end{equation*}
This relation is ensured if one has:
\begin{equation*}
\prod_{i=1}^{d-1}(s^2q^{-2i}-1) \equiv \frac{s^{2d}-1}{s^2-1}   \left(\text{modulo } \  \varphi_d(q^2) \right).
\end{equation*}
Up to a power of $q$, the left hand side term is given by $\prod_{i=1}^{d-1}(s^2-q^{2i})$. Now, if we look at the coefficients of $s$, we see that up multiples of $q^{2d}$ they are the same as the coefficients of $s$ in $\frac{s^{2d}-1}{s^2-1}$, which means that modulo $\varphi_d(q^2)$ the two polynomials are equal. This shows the formula for $\widetilde{\cJ}_{\cN}$ from the statement.
\end{proof}
\begin{lem}[Structural theorem of the maximal quotient ideal]\label{lem:idcy}
The ideal $\widetilde{\cJ}_\cN$ is equal to the ideal $\widetilde\cI_{\cN}$ given in \eqref{qideal}, i.e. it can be described via the product of the ideals associated to the set of divisors of $\cN$, as below:
\begin{equation*}
\widetilde{\cJ}_\cN
=
\widetilde{\cI}_\cN
=\left(\left(\varphi_\cN(q^2)\right) \cdot (s^2 q^{-2\cN+2} - 1)\right) \cap \prod_{\text{\scriptsize$\begin{matrix}
 d\mid \cN \\ d \neq 1, \cN \end{matrix}$}}
 \left(\varphi_d(q^2), \frac{s^{2d}-1}{s^2-1} \right).
\end{equation*}
 \end{lem}
\begin{proof} We recall the description of the ideal $\widetilde{\cJ}_\cN$ from Lemma \ref{strcy}:
\begin{equation*}
\widetilde{\cJ}_\cN
=\left((s^2 q^{-2\cN+2} - 1) \cdot\left(\varphi_\cN(q^2)\right)\right) \cap \bigcap_{\text{\scriptsize$\begin{matrix}
 d\mid \cN \\ d \neq 1, \cN \end{matrix}$}}
 \left(\varphi_d(q^2), \frac{s^{2d}-1}{s^2-1} \right).
\end{equation*}
We will show that:
\begin{equation*}
 \prod_{\text{\scriptsize$\begin{matrix}
 d\mid \cN \\ d \neq 1, \cN \end{matrix}$}}
 \left(\varphi_d(q^2),\frac{s^{2d}-1}{s^2-1} \right)
=
 \bigcap_{\text{\scriptsize$\begin{matrix}
 d\mid \cN \\ d \neq 1, \cN \end{matrix}$}}
 \left(\varphi_d(q^2),\frac{s^{2d}-1}{s^2-1} \right).
\end{equation*}
For this we are going to use the following criteria for a set of ideals, that ensures that their intersection is the same as the product. 
\begin{lem}
Let $R$ be a commutative ring and $I_1,...,I_n$ a set of ideals of $R$ such that any $I_i$ and $I_j$ are comaximal, meaning that:
$$I_i+I_j=R, \ \ \forall 1\leq i<j\leq n.$$
Then we have:
\begin{equation*}
\bigcap_{i=1}^n I_i=\prod_{i=1}^n I_i.
\end{equation*}
\end{lem}
We are going to use this criteria for our set of ideals. Let us consider $d_1\mid \cN, d_2\mid \cN$ such that $  1<d_1<d_2<\cN$.
We denote:
\begin{equation*}
\begin{aligned}
 \cI_{\cN,d_1}=\left(\varphi_{d_1}(q^2), \frac{s^{2d_1}-1}{s^2-1} )\right)\\
 \cI_{\cN,d_2}=\left(\varphi_{d_2}(q^2), \frac{s^{2d_2}-1}{s^2-1} )\right).
\end{aligned}
\end{equation*}
\begin{lem}\label{comax}
If $d_1 \neq d_2$, then $\varphi_{d_1}(q^2)$ and $\varphi_{d_2}(q^2)$ are coprime in $\Q[q^{\pm1}]$ and so the ideals $\left(\varphi_{d_1}(q^2)\right)$ and $\left(\varphi_{d_2}(q^2)\right)$ are comaximal.
\end{lem}
We remark that this property is not true if we are working over $\Z[q^{\pm2}]$, see \cite{TA}  for an explicit example.
 In turn, this means that  $ \cI_{\cN,d_1}=\left(\varphi_{d_1}(q^2), \frac{s^{2d_1}-1}{s^2-1} )\right)$ and  $ \cI_{\cN,d_2}=\left(\varphi_{d_2}(q^2), \frac{s^{2d_2}-1}{s^2-1} )\right)$ are comaximal in $\Liq$. This together with the previous Lemma ensures that
\begin{equation*}
\bigcap_{\text{\scriptsize$\begin{matrix}
 d\mid \cN \\ d \neq 1, \cN \end{matrix}$}}
 \left(\varphi_d(q^2), \frac{s^{2d}-1}{s^2-1}  \right)
=
  \prod_{\text{\scriptsize$\begin{matrix}
 d\mid \cN \\ d \neq 1, \cN \end{matrix}$}}
 \left(\varphi_d(q^2), \frac{s^{2d}-1}{s^2-1} \right)
\end{equation*}
which concludes the proof of this lemma. 
\end{proof}
%
%\bibliography{ref}{}
\bibliographystyle{plain}

 {
 
Universit\'e Clermont Auvergne, CNRS, LMBP, F-63000 Clermont-Ferrand, France, \\Institute of Mathematics “Simion Stoilow” of the Romanian Academy, 21 Calea Grivitei Street, 010702 Bucharest, Romania.}
\

{\itshape cristina.anghel@uca.fr, cranghel@imar.ro}
\par
{
Department of Mathematics, Faculty of Science and Engineering, 3-4-1 Ohkubo, Shinjuku-ku, Tokyo 156-844, Japan
}
\par
{\itshape murakami@waseda.jp}

\end{document}